\documentclass[12pt]{amsart}
\usepackage{amssymb,latexsym}
\usepackage{amsfonts}
\usepackage{amscd}
\usepackage{amsmath}
\usepackage{graphicx,array}

\theoremstyle{plain}
\numberwithin{equation}{section}
\newtheorem{theorem}{Theorem}[section]
\newtheorem{corollary}[theorem]{Corollary}

\newtheorem{lemma}[theorem]{Lemma}

\theoremstyle{definition}
\newtheorem{definition}[theorem]{Definition}

\newtheorem{remark}[theorem]{Remark}

\def\rmS{\mathrm{S}}
\def\rmU{\mathrm{U}}

\def\gep{\epsilon}

\def\gL{\Lambda}

\newcommand{\fa}{\mathfrak{a}}

\newcommand{\fg}{\mathfrak{g}}

\newcommand{\fk}{\mathfrak{k}}

\newcommand{\fq}{\mathfrak{q}}

\newcommand{\fu}{\mathfrak{u}}
\newcommand{\fz}{\mathfrak{z}}

\newcommand{\R}{\mathbb{R}}
\newcommand{\real}{\mathbb{R}}
\newcommand{\C}{\mathbb{C}}
\newcommand{\D}{\mathbb{D}}

\newcommand{\Z}{\mathbb{Z}}

\newcommand{\pr}{\mathop{\mathrm{pr}} }

\newcommand{\gl}{\lambda}

\newcommand{\eps}{\varepsilon}

\renewcommand{\Re}{\mathrm{Re}}
\renewcommand{\Im}{\mathrm{Im}}
\newcommand{\Exp}{\mathrm{Exp}}

\theoremstyle{plain}

\newcommand{\LK}{\Lambda^+(M)}

\newcommand{\Drc}{\overline{D}_r(o)}
\newcommand{\Bsc}{\overline{D}_s(o)}
\newcommand{\Bec}{\overline{D}_{\epsilon}(o)}

\newcommand{\SO}{\mathrm{SO}}
\newcommand{\Supp}{\mathrm{Supp}}
\newcommand{\PW}{\mathrm{PW}}

\newcommand{\LM}{\Lambda^+(M)}

\newcommand{\ra}{\rangle}
\newcommand{\la}{\langle}
\def\sideremark#1{\ifvmode\leavevmode\fi\vadjust{\vbox to0pt{\vss
 \hbox to 0pt{\hskip\hsize\hskip1em
\vbox{\hsize2cm\tiny\raggedright\pretolerance10000
 \noindent #1\hfill}\hss}\vbox to8pt{\vfil}\vss}}}%

\begin{document}
\setcounter{section}{0}
\title[Distributions on symmetric spaces]{Fourier transforms of spherical
distributions on compact symmetric spaces}
\author{Gestur \'Olafsson and Henrik Schlichtkrull}
\date{April 29, 2009}
\thanks{The research of
G. \'Olafsson was supported by NSF grants DMS-0402068 and DMS-0801010}
\begin{abstract} In our previous articles \cite{OS} and \cite{OS2}
we studied Fourier series on a symmetric space $M=U/K$
of the compact type. In
particular, we proved a Paley-Wiener type theorem for the smooth
functions on $M$, which have sufficiently small support and
are $K$-invariant, respectively $K$-finite.
In this article we extend these results to $K$-invariant distributions
on $M$. We show
that the Fourier transform of a distribution, which
is supported in a sufficiently small ball around the base point, extends
to a holomorphic function of exponential type. We describe the image
of the Fourier transform in the space of holomorphic functions.
Finally, we characterize the singular
support of a distribution in terms of its Fourier transform, and we
use the Paley-Wiener theorem to characterize the distributions
of small support, which are in the range of a
given invariant differential operator.
The extension from symmetric spaces of compact
type to all compact symmetric spaces is sketched in an appendix.
\end{abstract}
\maketitle

\section*{Introduction}
\noindent
The Paley-Wiener theorem for $\real^n$ describes
(in the version due to L.~Schwartz) the image by
the Fourier transform of the space of compactly supported
smooth functions on $\real^n$. A similar theorem describes
the image of the space of compactly supported distributions.
More precisely, let $C_c^\infty(\real^n)$ and
$C^{-\infty}_c(\real^n)$ denote the spaces of
of compactly supported smooth functions and distributions,
respectively.
Then the Fourier image of $C_c^\infty(\real^n)$ is the space
of entire functions $F$ on $\C^n$ of exponential type, that is,
for which there exist $r>0$
and for every $N\in \Z^+=\{0,1,2,\ldots \}$ a constant $C_N$ such that
$$|F(\lambda )|\le C_N (1+|\gl |)^{-N}e^{r|\mathrm{Im}(\lambda)|}$$
for all $\lambda \in \C^n$. Furthermore, the Fourier image
of $C_c^{-\infty}(\real^n)$ is the space
of entire functions $F$ for which there exist $r>0$ and for
some $N\in \Z^+$ a constant $C_N$ such that
$$|F(\lambda )|\le C_N(1+|\gl |)^{N}e^{r|\mathrm{Im}(\lambda)|}$$
for all $\lambda \in \C^n$. An important aspect of these
theorems is that the smallest exponent $r$
in the estimates matches with the radius of the smallest
closed ball $\bar B_r=\{x\in\real^n\mid |x |\leq r\}$ containing the
support of the function or distribution.
Finally, by an analogous result due to
H\"ormander, the singular support of a compactly supported
distribution is contained in $\bar B_r$
if and only if its Fourier transform $F$ satisfies the following
condition. There exists $N\in \Z^+$ and for every
$m\in\Z^+$ a constant $C_m$ such that
$$|F(\lambda )|\le C_m(1+|\gl |)^{N}e^{r|\mathrm{Im}(\lambda)|}$$
for all $\gl\in\C^n$ with $|\mathrm{Im}(\lambda)|\leq
m\log(1+|\gl|)$. See \cite{Horm}, Section 1.7.

There are several generalizations of these theorems to
settings where $\real^n$ is replaced by a symmetric space $X$.
The most general results have been obtained for $X=G/K$ a
Riemannian symmetric space of the non-compact type, by
Gangolli \cite{Ga} and Helgason \cite{He1,He2}
for smooth functions, and by Eguchi, Hashizume and Okamoto \cite{EHO}
for distributions. Again the exponent matches with the
radius of the support. For functions this is seen in the
references just mentioned, and for distributions it is shown
by Dadok \cite{Dadok}, who gives a proof of the distributional
Paley-Wiener theorem different from that of \cite{EHO}.
Combining these results with H\"ormander's theorem
for $\real^n$, a characterization of singular
supports is easily deduced, see \cite{Dadok}
(see also \cite{AK}).

In the present paper we investigate the generalization of the
theorems for distributions to a Riemannian symmetric space $M=U/K$
of the compact type. In previous papers \cite{OS} and
\cite{OS2}, we have treated the case of $K$-invariant, respectively,
$K$-finite smooth functions on $M$. These papers generalized
partial results in \cite{BOP, Camp, Gon, TomK}.
In contrast to the non-compact cases $G/K$,
the results obtained for $M$ are \textit{local} in the sense that
they are only valid for functions supported in sufficiently
small balls, with an explicit (but not necessarily optimal)
upper bound for the radius.
In the present paper we use the results of \cite{OS} to extend the
$K$-invariant Paley-Wiener theorem to $K$-invariant distributions
on $M$, including the analogous result for singular support.
The more general case of $K$-finite distributions
can be treated similarly, based on \cite{OS2}
(the details are omitted). In an appendix at the end of the paper
we briefly discuss the extension to all compact symmetric spaces
of the results from \cite{OS} as well as those of the present paper.

The Paley-Wiener theorems have also been generalized to
non-Riemannian symmetric spaces. General reductive
symmetric spaces are treated in \cite{BanS} and \cite{BanS2}.
The case of a reductive Lie group (which can be considered as
a symmetric space), was earlier treated in \cite{Arthur},
see also \cite{BSInd} and \cite{D}.
Hyperbolic spaces were treated in \cite{A}. Some partial
results have been obtained for the Fourier-Laplace
transform on causal symmetric spaces, see \cite{AO, AOS, OP1}
and the overview in \cite{OlPasq}.

The article is organized as follows.
In Section \ref{s: notation} we introduce the
basic notation. In Section  \ref{s: Fourier} we
discuss the parametrization of the irreducible
unitary $K$-spherical representations and the related
Fourier transform.
Let $f$ be a $K$-spherical smooth function (or distribution)
on $M$, then its Fourier transform $\widetilde{f}(\gl)$
is defined for $\gl$ in the semi-lattice $\Lambda^+(M)$
consisting of the highest weights $\lambda\in\fa^*_\C$
of irreducible $K$-spherical representations (the weights are purely
imaginary linear forms on the maximal abelian subspace $\fa$).
In  Section \ref{s: main thm} we recall the main
result of \cite{OS} which, in short, says the following.
Assume $f$ is smooth, then the Fourier transform
$\widetilde{f}$ extends to
a holomorphic function on $\fa_\C^*$
of exponential type, and the best exponent of growth
is equal to the radius of the smallest closed ball
around the origin, which contains the support of $f$.
Here it is required that
the support of $f$ is sufficiently small, as explained
in the Remark 4.3 in \cite{OS}.

Section \ref{s:distributions} contains the main results of this article.
First, we introduce the Paley-Wiener space $\PW_r^*(\fa)$ of
holomorphic functions on $\fa_\C^*$ such that:
\begin{itemize}
\item[a)] There exists $k\in\Z^+$ and $C>0$ such that
for all $\lambda\in\fa_\C^*$,
$$|\Phi (\gl)|\leq C(1+|\gl|)^{k} e^{r|\Re\gl |}\, $$
\item[b)] $\Phi (w(\gl+\rho)-\rho)=\Phi(\gl)$ for
all $w$ in the Weyl group $W$.
\end{itemize}

Let $ C^{-\infty}_r(M)^K$ denote the space of $K$-invariant distributions
on $M$ with support  in a closed ball of radius $r$
around the origin.
Our main results are:
\medskip

\noindent
\textbf{Theorem \ref{t: PW-d}}
\textrm{(Local Paley--Wiener theorem for distributions)}
\textit{There exists
$R>0$ such that the following holds
for each $0<r<R$.
}

\begin{itemize}
\item[i)] \textit{Let $F\in C^{-\infty}_r(M)^K$.
Then the Fourier transform $\tilde{F} \colon \Lambda^+(M)\to\C$
extends  to a function in $\PW_r^*(\fa)$.}

\item[ii)] \textit{Let $\Phi \in\PW_r^*(\fa)$. There exists
a unique distribution $F\in C^{-\infty}_r(M)^K$ such that
$\tilde{F} (\mu)=\Phi(\mu)$ for all $\mu\in \Lambda^+(M)$.}

\item[iii)] \textit{The functions in the Paley-Wiener space $\PW_r^*(\fa )$
are uniquely determined by their values on $\Lambda^+(M)$.}
\end{itemize}
\textit{Thus,
the Fourier transform followed by the extension gives a bijection}
$$C_r^{-\infty}(M )^K\simeq \PW_r^*(\fa )\, .$$

\smallskip
\noindent
\textbf{Theorem \ref{t:singsupp}}
\textrm{(Characterization of singular support)}
\textit{Let
$R$ be as above and let
$0<s\le r<R$. Let $F\in C^{-\infty}_r(M)^K$. Then the singular
support of $F$ is contained in a closed ball of radius $s$ if and only if there exists
$N\in \Z$ such that for each $m\in \Z^+$ the holomorphic
extension of $\tilde{F}$ satisfies
$$|\tilde{F}(\lambda)|\le C_m (1+|\lambda |)^N e^{s|\Re \lambda |}$$
for all $\lambda \in \fa_\C^*$ with $|\Re \lambda |\le m \log
(1+|\lambda |)\, .$
}

\medskip

One of the consequences of the Paley-Wiener theorem is a condition
for the solvability of invariant differential equations $DT=F$
in the space $C^{-\infty}_r(M)^K$.
This is stated in Theorem \ref{t:solvability}.

\section{Basic notation}
\label{s: notation}
\noindent
Let $M$ be a connected Riemannian symmetric space of the compact type.
Then there exists a compact connected semisimple Lie group $U$
acting on $M$ by isometries
and a closed subgroup $K\subset U$ such that $M=U/K$. Furthermore, there
exists an involution $\theta : U\to U$ such that
$U^\theta_0\subset K\subset U^\theta$. Here $U^\theta$ denotes the
subgroup of $\theta$-fixed points, and
$U^\theta_0$ its identity component. We denote the base point $eK$ in
$M$ by $o$.

Let $\fu$ denote the Lie algebra of $U$, then $\theta$
induces an involution of $\fu$ (denoted by the same symbol).
Let $\fu=\fk\oplus\fq$ be the corresponding decomposition
in eigenspaces for $\theta$. Let $\la\,\cdot\,,\,\cdot\,\ra$
be the inner product on $\fu$ defined by
$\langle X,Y\rangle=-B(X,Y)$, where $B$ is the Killing
form.
We assume that the Riemannian metric $g$ of $M$ is
normalized such that it agrees with $\langle\,\cdot\,,\,\cdot\,\rangle$
on the tangent space $\fq\simeq T_oM$.
We denote by $\exp$ the exponential map $\fu\to U$,
and by $\Exp$ the map $\fq\to M$
given by $\Exp(X)=\exp(X) \cdot o$.
Denote by $ B_r(0)$ the open ball in $\fq$ of radius $r>0$ and centered
at $0$ and  $D_r(o)$ the open metric ball in $M$ of radius $r>0$ and centered at
$o$. Similarly $\bar B_r (0)$ and $\bar D_r (o)$ stand for the closed balls.
The exponential map
$\Exp$ is surjective and an analytic diffeomorphism $B_r(0)\to D_r(o)$
for $r $ sufficiently small.

Let $\fa\subset\fq$ be a maximal abelian subspace,
$\fa^*$ its dual space, and $\fa^*_\C$ the complexified
dual space. Then $\langle \cdot ,\cdot \rangle$ defines an inner product on
$\fa^*$. By sesquilinear extension we obtain inner products on $\fa_\C^*$
and $i\fa^*$, which we shall denote by the same symbol.
The corresponding norm is denoted by
$|\cdot |$.

We denote by $\Sigma$  the set of non-zero
(restricted) roots of
$\fu_\C$ with respect to $\fa_\C$. Then $\Sigma\subset i\fa^*$.
Furthermore,
$\Sigma^+$  stands for a fixed set of positive roots and
$\rho\in i\fa^*$ denotes half of the sum of the roots in $\Sigma^+$ counted
with multiplicities.
The corresponding Weyl group, generated by the reflections
in the roots, is denoted $W$.

\section{Fourier analysis on $M$}
\label{s: Fourier}
\noindent
In this section we recall the basic facts on Fourier series
on $M$ and  the parametrization
of $K$-spherical representations.
Let $\pi$ denote an irreducible unitary representation
of $U$, and $V_\pi$ the Hilbert space on which $\pi$ acts.
Let
$$V^K_\pi=\{v\in V_\pi\mid (\forall k\in K)\, \pi(k)v=v\}\, .$$
If $V^K_\pi\not= \{0\}$ then $\dim V^K_\pi=1$ and $\pi$ is
said to be $K$-\textit{spherical}. If $\pi$ is $K$-spherical, then
$e_\pi$ will denote a fixed choice of a unit vector in $V^K_\pi$.

Recall the
following parametrization of $K$-spherical irreducible
representations of $U$, due to Helgason (see \cite{GGA}, p.\ 535).
Denote by $\tilde{U}$ the universal covering of $U$
and by $\kappa$ the canonical projection $\tilde{U}\to U$.
Then $\theta$ defines an involution $\tilde\theta$ on $\tilde{U}$, and
the group $\tilde{K}$ of $\tilde\theta$-fixed points
is connected.
If $\pi$ is a $K$-spherical representation of $U$, then $\pi\circ
\kappa$ is a $\tilde{K}$-spherical
representation of $\tilde{U}$.

\begin{theorem}
\label{t: Helgason classification} The map $\pi \mapsto \mu$, where
$\mu\in i\fa^*$
is the highest weight of $\pi$, induces a bijection between the set of
equivalence classes of irreducible $\tilde{K}$-spherical
representations of $\tilde{U}$ and the set
\begin{equation}
\label{e: Helgason condition}
\Lambda^+(\tilde{U}/\tilde{K})
=\{\mu\in i\fa^*\mid (\forall \alpha \in\Sigma^+)\,\,
\frac{\langle\mu,\alpha\rangle}{\langle\alpha,\alpha\rangle}
\in\Z^+\}\, .
\end{equation}
\end{theorem}

For $\mu\in \Lambda^+(\tilde U/\tilde K)$,
let $(\pi_\mu,V_\mu)$ denote a fixed irreducible
unitary representation of $\tilde U$ with highest
weight $\mu$, and let $e_\mu=e_{\pi_\mu}$.
We denote by $\Lambda^+(U/K)$ the set of elements in
$\Lambda^+(\tilde{U}/\tilde{K})$,
for which the representation $\pi_\mu$ of $\tilde U$
descends to a representation of $U$ with a $K$-fixed
vector. Note that if it descends, it will have a $K_0$-fixed,
but not necessarily a $K$-fixed vector. This
was not made clear in \cite{OS}, Theorem 3.1, which
is only valid as stated
under the extra condition that $K$ is connected.
As an example take
$U=\SO (n)$, $n\ge 3$, and
$$K=\mathrm{O}(n-1)=\left\{\left. \begin{pmatrix} \det (A) & 0\cr
0 &A\end{pmatrix}\right| A\in O(n-1)\right\}\, .$$
Then $K_0\simeq \SO (n-1)$. The natural representation of $\SO (n)$ acting
on $\C^n$ has a $K_0$-fixed vector $e_1$, but is not $K$-spherical.
However, the restricted validity does not affect the main
results of \cite{OS}, as the exact description of $\Lambda^+(U/K)$ is
not used. However, the following property of $\Lambda^+(U/K)$ is used.

\begin{lemma}\label{semilattice}
The subset $\Lambda^+(U/K)\subset i\fa^*$ is closed under addition,
and it has full rank in $i\fa^*$,
that is, there exist $\mu_1,\dots\mu_n\in i\fa^*$
linearly independent where $n=\dim\fa$,
such that
\begin{equation}\label{sublattice}
\Z^+\mu_1+\dots+\Z^+\mu_n\subset \Lambda^+(U/K).
\end{equation}
\end{lemma}

\begin{proof} It follows from Theorem \ref{t: Helgason classification}
that the result holds for
$\Lambda^+(\tilde U/\tilde K)$. In fact, in this case equality
is attained in (\ref{sublattice}) when $\mu_1,\dots,\mu_n$
are the fundamental weights (see \cite{Vretare}).

Let $K^*=\kappa^{-1}(K)$ then $\tilde K=K^*_0$,
and the quotient $K^*/\tilde K$ is a finite group, which
acts by a homomorphism $\gamma_\mu\colon K^*/\tilde K\to\C$
on the one-dimensional space $V^{\tilde K}_{\pi_\mu}$
for each $\mu\in\Lambda^+(\tilde{U}/\tilde{K})$.
In particular, we see that $\mu$ belongs to
$\Lambda^+(U/K)$ if and only if $\gamma_\mu$ is trivial.
We shall see below that
\begin{equation}\label{tensorproduct}
\gamma_{\mu+\nu}=\gamma_\mu\cdot\gamma_\nu
\end{equation}
for all $\mu,\nu\in\Lambda^+(\tilde{U}/\tilde{K})$.
It follows that
$\Lambda^+(U/K)$ is closed under addition, and also that
$\gamma_{p\mu}=1$, where $p$ is the
order of $K^*/K$. Hence $p\Lambda^+(\tilde U/\tilde K)\subset
\Lambda^+(U/K)$, and thus
the lemma follows from (\ref{tensorproduct}).

For each $\mu\in \Lambda^+(\tilde U/\tilde K)$, let
$e_\mu=e_{\pi_\mu}\in V_\mu=V_{\pi_\mu}$ denote the chosen
$\tilde K$-fixed unit vector, and let $v_\mu\in V_\mu$ be a
highest weight vector normalized such that $\langle v_\mu, e_\mu
\rangle =1$. Then $\int_{\tilde K} \pi_\mu(k)v_\mu\,dk=e_\mu$.

Consider the tensor product $V_\mu\otimes V_\nu$. It is
well known that the representation $V_{\mu+\nu}$ occurs
with multiplicity one in the tensor product, and that
$v_\mu\otimes v_\nu$ is a highest weight vector in $V_{\mu+\nu}$.
The vector
$$e:=\int_{\tilde K} \pi_\mu(k)v_\mu\otimes \pi_\nu(k)v_\nu\, dk
\in V_{\mu+\nu}$$
is $\tilde K$-fixed.
Using Fubini's theorem and the invariance of Haar measure, we see that
\begin{equation}\label{emuplusnu}
\int_{\tilde K} (\pi_\mu(l)\otimes 1) e\,dl=e_\mu\otimes e_\nu.
\end{equation}
In particular, $e\neq 0$ and we can identify $e$ as a multiple of
the unit vector $e_{\mu+\nu}$.
The desired relation (\ref{tensorproduct}) follows from
(\ref{emuplusnu}), by
using that Haar measure on $\tilde K$ is invariant
under the adjoint action of $K^*$.
\end{proof}

For $\mu\in \LK=\Lambda^+(U/K)$
the  \textit{spherical function} associated
with $\mu$ is the matrix coefficient
$$\psi_\mu(x)=(\pi_\mu(x)e_\mu,e_\mu), \quad x\in U\, .$$
It is left and right $K$-invariant and can therefore be viewed as
a left $K$-invariant function on $M$. It is an eigenfunction
of $\mathbb{D}(M)$, the algebra of invariant differential
operators on $M$. The
{\it spherical Fourier transform} of a
$K$-invariant $L^1$-function $f$ on
$M$ is the function $\tilde f: \gL^+(M)\to \C$ defined by
$$\tilde f(\mu)=\int_M f(x)\overline{\psi_\mu(x)}\,dx =(f,\psi_\mu),$$
where $dx$ is the normalized invariant measure on $M$ (that is,
$\int_M \, dx=1$). Notice that if $f\in L^p(M)$ then $|\widetilde{f}(\mu )|
\le \|f\|_p$ as $|\psi_\mu(x)|\le 1$. In particular
$$|\tilde f(\mu)|\leq \|f\|_\infty$$
if $f$ is continuous, and hence bounded.
It follows from the Schur orthogonality relations that
\begin{equation}\label{Schur}
\widetilde\psi_\nu(\mu)=\delta_{\nu,\mu} d(\mu)^{-1}
\end{equation}
for $\nu,\mu\in\Lambda^+(M)$, where
$d(\mu)=\dim(V_\mu)$.

The {\it spherical Fourier series}
for $f$ is the series given by
\begin{equation}\label{Fourier series}
\sum_{\mu\in\gL^+(M)} d(\mu)\tilde f(\mu)\psi_\mu\, .
\end{equation}

Denote by $\Delta_M$ the negative definite Laplace operator on $M$.
Then
\begin{equation}\label{eq-eigenvalue}
\Delta_M\psi_\mu =- \langle \mu , \mu +2\rho \rangle \, \psi_\mu \, .
\end{equation}
Based on (\ref{eq-eigenvalue}) it can be shown,
see Sugiura \cite{Sugiura}, that
$f$ is smooth if and only if the Fourier
transform $\tilde f$ is {\it rapidly decreasing}, that is,
for each $k\in\Z^+$ there exists a constant $C_k$ such
that
\begin{equation}\label{eq-growthFTfunctions}
|\tilde f(\mu)|\leq C_k (1+|\mu|)^{-k}
\end{equation}
for all $\mu\in\gL^+(M)$. In this case
the Fourier series
(\ref{Fourier series}) converges pointwise and absolutely
to $f$.

There are different ways to describe the topology on
$C^\infty (M)^K$. First,  the topology on
$C^\infty (U)$ is defined by the seminorms
\begin{equation}\label{eq-semiNormNu}
\nu_p(f):=\| L_p f\|_\infty
\end{equation}
where $p\in  U(\fg )$. Here $L$ denotes the left regular
representation and $U(\fg )$ the universal
enveloping algebra of $\fg$.
If $\mathcal{C}$ is a
closed subspace of $C^\infty (U)$ then the topology on $\mathcal{C}$ is
given by the same family of seminorms. This applies to the
space $C^\infty (M)$, viewed as the space of \textit{right} $K$-invariant
smooth functions on $U$, the space $C^\infty (M)^K$ of \textit{left} $K$-invariant
functions in $C^\infty (M)$, as well as the spaces $C^\infty_r(M)$ and
$C^\infty_r(M)^K$, $r>0$, where the subscript $r$ indicates that the support
is contained in $\bar D_r (o)$. Note, if $r$ is big enough then $C^\infty_r(M)=C^\infty (M)$.

According to \cite{Sugiura} the topology can also be described
using $\Delta_M$:
\begin{lemma}\label{sugiura1} The topology of $C^\infty (M)$,
$C^\infty (M)^K$, and $C_r^\infty (M)^K$ is given by the seminorms
$$\tau_m (f)=\|\Delta_M^m f\|_\infty\, ,\quad m\in\Z^+\, .$$
\end{lemma}
\begin{proof} This is the corollary to Theorem 4 in \cite{Sugiura}.
\end{proof}

We shall also need the following fact from \cite{Wallach} Lemma 5.6.7
or \cite{Sugiura} Lemma 1.3.

\begin{lemma}\label{le-su}  There exists $t_0\in\real$ such that
$\sum_{\mu\in\Lambda^+(M)}(1+|\mu|)^{-t}<\infty$
if $t>t_0$.
\end{lemma}

By Weyl's dimension formula, the map
$\mu \mapsto
d(\mu)$ extends to a polynomial function on $\fa_\C^*$.
We derive the following
consequence from Lemmas
\ref{sugiura1} and \ref{le-su}
together with
(\ref{eq-eigenvalue}) and
(\ref{eq-growthFTfunctions}).

\begin{lemma}\label{Fourier convergence} Let $f\in C^\infty(M)$.
Then the Fourier series
{\rm (\ref{Fourier series})}
converges to $f$ in $C^\infty(M)$.
\end{lemma}

It follows from the $KAK$-decomposition of $U$ that
the restriction map $f\mapsto f|_{A\cdot o}$ is injective
for $f\in C^\infty(M)^K$. We use the topology on $C^\infty (A\cdot o)$
given by the seminorms $\nu_u(f)=\|L_uf\|_\infty$, $u\in U(\fa )$.
Then $C^\infty (A\cdot o)$ is a Fr\'echet space, and
$C^\infty (A\cdot o)^W$ is a closed subspace whose topology
is given by the same family of seminorms.
The following lemma gives a different way to describe the topology on
$C^\infty (M)^K$:

\begin{lemma}\label{le-restriction} The restriction map from
$C^\infty (M)^K$ to $C^\infty(A\cdot o)^W$ is
a topological isomorphism. It
is also a topological isomorphism from
$C^\infty_r (M)^K$ onto $C^\infty_r(A\cdot o)^W$,
for each $r>0$.
\end{lemma}
\begin{proof}
According to \cite{Dadok2}, Theorem 1.7, the restriction map is
bijective. It is obviously continuous. By
the open mapping
theorem for Fr\'echet spaces \cite{Treves}, Theorem 17.1., p. 170, it follows that
the restriction map is a topological isomorphism. For the last statement
we note first that $C_r^\infty (A\cdot o)^W$ is
closed in $C^\infty (A\cdot o)^W$ and similarly for
$C^\infty_r(M)^K$ in $ C^\infty (M)^K$.
Furthermore, the metric ball in $A\cdot o$ of radius $r$
centered at $o$ is $\bar D_r (o)\cap A\cdot o$  and
$\bar D_r (o)=K (\bar D_r (o)\cap A\cdot o)$. Hence
$f\mapsto f|_{A\cdot o}$ is a bijection
$C_r^\infty (M)^K\to C_r^\infty (A\cdot o)^W$,
and it follows from the first statement that it is an isomorphism.
\end{proof}

\section{The Fourier series of a distribution}\label{s: dist}
\noindent
The continuous dual of $C^\infty (M)$, denoted by
$C^{-\infty} (M)$, is the space of \textit{distributions} on $M$.
Recall that
$U$ acts on $C^{-\infty}(M)$ by
$$L_gF (f) =F (L_{g^{-1}}f )\,   , \quad g\in U,\,\,
f\in C^\infty (M)\, , \, \, \mathrm{and}\,\, F\in C^{-\infty}(M)\, .$$

Denote by $C^{-\infty}(M)^K$ the space of $K$-invariant distribution
on $M$. Since $C^\infty(M)^K$ is a closed subspace of $C^\infty(M)$,
we obtain a map from $C^{-\infty}(M)$ to $[C^{\infty }(M)^K]^*$,
by taking restrictions of linear forms to this subspace.
Here $[C^{\infty }(M)^K]^*$ denotes
the space of continuous linear forms on $C^{\infty }(M)^K$.
We provide $C^{-\infty}(M)$ and $[C^\infty (M)^K]^*$ with the
weak $*$-topology. The inclusion map $C^\infty (M)^K\hookrightarrow C^\infty (M)$ is
continuous. Hence the above projection is also continuous.

\begin{lemma}\label{le-distributionrestriction} The restriction
defines a linear isomorphism
$$C^{-\infty} (M)^K\simeq [C^{\infty }(M)^K]^*.$$
\end{lemma}

\begin{proof} Let
$\pr : C^\infty (M)\to C^\infty (M)^K$ be the projection
$\pr (f)(x)=\int_K f(kx)\, dk$. It is continuous, hence
the transposed $\pr^t$ maps
$[C^\infty ( M)^K]^*\to C^{-\infty }(M)$ is also continuous.
It  is easily seen that this provides the inverse to the restriction.
\end{proof}

\begin{lemma}\label{le-KinvariantDistribution} Let $F: C^\infty (M)^K\to \C$ be
linear. Then the following statements are equivalent:
\begin{enumerate}
\item $F$ is a $K$-invariant
distribution.
\item There exist $C>0$ and $m\in \Z^+$ such
that
\begin{equation}\label{eq-dist1}
|F(f)|\le C \max_{j=1,\ldots ,m} \|\Delta_M^j f\|_\infty  \qquad
(\forall f\in C^\infty(M)^K).
\end{equation}
\item
There exist $C> 0$
and finitely many $u_1,\ldots ,u_s \in U(\fa)$ such that
\begin{equation}\label{eq-dist}
|F(f)|\le C \max_{j=1,\ldots ,s} \|L_{u_j}(f|_{A\cdot o})\|_\infty   \qquad
(\forall f\in C^\infty(M)^K).
\end{equation}
\end{enumerate}
\end{lemma}
\begin{proof} This follows from Lemmas \ref{sugiura1} and
\ref{le-restriction}.
\end{proof}

Let $w^*\in W$ be such that $w^*(\Sigma^+)=-\Sigma^+$. Then
$\mu \mapsto \mu^*:=-w^*(\mu)$ defines a bijection of $\Lambda^+(M)$,
such that $\pi_{\mu^*}$ is the contragradient representation to $\pi_\mu$.
Notice that
$\overline{\psi_\mu}=\psi_{\mu^*}=\psi_\mu^\vee$ where $f^\vee
(g)=f(g^{-1})$. Furthermore
$d(\mu^*)=d(\mu)$.
We define the Fourier transform of a
spherical distribution
$F\in C^{-\infty} (M)^K$  by
\begin{equation}\label{eq-FTdistribution}
\tilde{F}(\mu):=F(\psi_{\mu^*})=F(\psi_\mu^\vee )\, .
\end{equation}
In particular, for smooth $K$-invariant functions regarded
as distributions by means of the pairing with the
invariant measure, the two notions of Fourier transform agree.

\begin{lemma} Let $F\in C^{-\infty} (M)^K$.
Then $\mu \mapsto \widetilde{F}(\mu)$ has at most
polynomial growth.
\end{lemma}
\begin{proof}
This follows from (\ref{eq-dist1}) and (\ref{eq-eigenvalue}).
\end{proof}

We can now write down the Fourier series for $F$.

\begin{lemma}\label{l:F(f)}
Let $F\in C^{-\infty }(M )^K$ and $f\in C^\infty (M )^K$.
Then
\begin{equation}\label{e:F(f)}
F (f) =\sum_{\mu \in \Lambda^+( M)} d (\mu) \tilde{f}(\mu^* )\tilde{F}(\mu )\,
\end{equation}
with absolute convergence. In particular, the distributional
Fourier transform $F\mapsto\tilde F$
is injective.
\end{lemma}
\begin{proof}
It follows from Lemma \ref{Fourier convergence}
that
$$f=\sum_{\mu \in \Lambda^+(M)} d (\mu^* ) \tilde{f}(\mu^* )\psi_{\mu^*}$$
in the topology of $C^\infty (M)^K$.
Since $F$ is continuous we can apply it termwise, and
since
$d (\mu ^*) =d(\mu )$
we then obtain (\ref{e:F(f)}) with convergence in $\C$.
The absolute convergence follows from
Lemma \ref{le-su}, since
$d(\mu)$ and $\tilde F(\mu)$ have at most polynomial
growth with respect to $\mu$.
\end{proof}

\section{Local Paley-Wiener Theorem for $K$-invariant functions on $M$ }
\label{s: main thm}
\noindent
We recall the main results from \cite{OS}.

\begin{definition}
\label{d: PW space}
(Paley-Wiener space)
For $r>0$ let
$\PW_r(\fa)$
denote the space of holomorphic functions $\varphi$ on $\fa_\C^*$
satisfying the following.
\begin{itemize}
\item[a)] For each $k\in\Z^+$ there exists a constant $C_k>0$ such that
$$|\varphi(\gl)|\leq C_k(1+|\gl|)^{-k} e^{r|\Re\gl|}$$
for all $\gl\in\fa_\C^*.$

\item[b)] $\varphi(w(\gl+\rho)-\rho)=\varphi(\gl)$ for
all $w\in W$, $\gl\in\fa_\C^*$.
\end{itemize}
\end{definition}

The following is Theorem 4.2 of \cite{OS}. As pointed
out in \cite{OS}, Remark 4.3, the known value for the
constant $R$ can be different in each part of the theorem.

\begin{theorem}
\label{t: PW}
There exists $R>0$ such that the following holds
for each $0<r<R$.

\begin{itemize}
\item[i)] Let $f\in C^\infty_r(M)^K$.
Then the Fourier transform $\tilde f\colon \Lambda^+(M)\to\C$
extends to a function in $\PW_r(\fa)$.
\item[ii)] Let $\varphi\in\PW_r(\fa)$. There exists
a unique function $f\in C^\infty_r(M)^K$ such that
$\tilde f(\mu)=\varphi(\mu)$ for all $\mu\in \Lambda^+(M)$.
\item[iii)] The functions in the Paley-Wiener space $\PW_r(\fa )$
are uniquely determined by their values on $\LM$.
\end{itemize}
Thus, the Fourier transform followed by the extension gives a
bijection
$$C^\infty_r (M)^K\simeq \PW_r(\fa )\, .$$
\end{theorem}

\begin{remark}
The proof of this theorem in \cite{OS} is not entirely correct,
as an error occurs in the proof of Corollary 10.2.
The function $\psi(\lambda)=\varphi(\lambda_1)\varphi_m(\lambda_2)$,
constructed in the proof of the corollary
is not of exponential type $r$ as stated, but only
of type $\sqrt{2}r$. This follows from the estimate
$|\lambda_1|+|\lambda_2|\leq \sqrt{2}|\lambda|$,
which is sharp. However, one can apply the theorem
of \cite{Cow} to construct an entire function $\psi$
on ${\mathfrak h}^*_\C$, which is of the proper
exponential type, and which restricts to $\varphi$
on $\mathfrak a$. The rest of the proof is then unchanged.
\end{remark}

\section{Analytic continuation of spherical functions}
\label{s: anal cont}
\noindent
We need some details from \cite{OS} concerning
the analytic continuation of
the spherical functions $\psi_\mu$ with respect to
the parameter $\mu$.

Let $\bar\Omega$ be the closure of
\begin{equation}\label{Omega}
\Omega=\{ X\in\fa \mid ( \forall\alpha\in\Sigma )\,\,
|\alpha(X)|<\frac \pi2\}\, .
\end{equation}
As $U$ is compact, it follows that $U$ is contained
in a complex Lie group $U_\C$ with Lie algebra
$\fu_\C$. Denote by $K_{0,\C}$ the analytic subgroup of
$U_\C$ corresponding to $\fk_\C$. Note that we are at this point not assuming
that $\theta$ extends to an involution on $U_\C$.
Let $\fg=\fk\oplus i\fq\subset \fu_\C$ and let $G$ be the corresponding analytic
subgroup of $U_\C$, then $K_0\subset G$ is a maximal compact subgroup.
The space
$M^d=G/K_0$ is the (noncompact) dual of $U/K_0$. Note that the center
of $G$ is contained in $K_0$ so $M^d$ is independent of the choice of
the complexification $U_\C$.
Let $K_\C = K K_{0,\C}$.
Then $M,M^d
\subset M_\C := U_\C /K_\C$. Then $K_\C$ is a closed subgroup of
$U_\C$.
For each $\mu\in \Lambda^+(M)$ the spherical function $\psi_\mu$ has an
analytic continuation to $M_\C=U_\C/K_{\C}$, denoted by the same symbol, and
$\psi_\mu|_{M^d}=\varphi_{\mu + \rho}$ where $\varphi_\lambda$ denotes
the spherical function on $M^d$ with spectral parameter $\lambda$.
According to \cite{KrSt,Opd} (see also the proof due to J. Faraut in
\cite{BOP}) the spherical function $\varphi_\lambda$ on $M^d$
has a holomorphic extension as a $K_{\C}$-invariant function on
$K_{ \C} \exp(2\Omega )\cdot o$ for every
$\lambda\in\fa^*_\C$.
For each $x\in K_{ \C} \exp(2\Omega )\cdot o$ and
$\gl\in\fa_\C^*$ we define
\begin{equation}
\psi_\gl(x)=\varphi_{\gl + \rho}(x)
\end{equation}
and thus obtain an extension to $\fa_\C^*$ of
the map $\mu\mapsto \psi_\mu(x)$ where
$\mu\in\Lambda^+(M)$.
The map $(\lambda ,x ) \mapsto \psi_\gl(x)$ is
holomorphic on the open set $\fa_\C^*\times \exp(\Omega+i\fa)\cdot
o\subset \fa_\C^*\times A_\C\cdot o$
and it satisfies the following estimate, cf.  \cite{Opd}, Proposition 6.1:

\begin{lemma}\label{t:estimate}
There exists a constant $C$ such that
\begin{equation}\label{eq-t:estimate1}
|\psi_\gl(\exp (X+iY)\cdot o)|\leq
C \, e^{\max_{w\in W}\Re w\gl(X)-\min_{w\in W}\Im w\gl (Y)}
\end{equation}
for all $X\in \bar{\Omega}$, $Y\in\fa$ and $\gl\in\fa_C^*$.
\end{lemma}

\begin{corollary}\label{l:estimate derivatives}
Let $X_1,\dots,X_j\in\fa$ and $X\in \Omega$.
There exists a constant $C$ such that
\begin{equation}
|L_{X_1\dots X_j}\psi_\gl(\exp (X)\cdot o)|
\le C \, (1+|\gl|)^{j} e^{|X| |\Re \lambda |}
\end{equation}
for all $\gl\in\fa^*_\C$. The constant $C$ depends
locally uniformly on $X$.
\end{corollary}

\begin{proof} Let $V$ be a complex neighborhood of $0$
such that
$X+V\subset \Omega +i\fa$. Let $Z_1,\ldots , Z_\ell$
be an
orthonormal basis for $\fa$. By considering linear
combinations and using that $\fa$ is abelian it is
enough to prove the claim for the derivatives
$$\left(\frac{\partial\, }{\partial x_1}\right)^{m_1}\ldots \left(\frac{\partial\, }{\partial x_\ell
}\right)^{m_\ell}\psi_{\gl}(\exp(X+x_1Z_1+\ldots +x_\ell Z_\ell)\cdot o)|_{x_1=\ldots =x_\ell =0}\, .$$

To simplify the notation let
$$f_\lambda (x_1,\ldots ,x_\ell)=\psi_{\gl}(\exp (X+x_1Z_1+\ldots +x_\ell Z_\ell)
\cdot o)\, .$$
We will also use the following notation for $m=(m_1,\ldots ,m_\ell)\in (\Z^+)^\ell$
and $\zeta =(\zeta_1,\ldots ,\zeta_\ell)\in \C^\ell$:
$m!:=m_1!\cdots m_\ell !$, $|m|=m_1+\ldots +m_\ell$,
$m+1=(m_1+1,\ldots ,m_\ell +1)$, and $\zeta^m=\zeta_1^{m_1}\ldots \zeta_\ell^{m_\ell}$.
We also set $\partial^m=(\partial /\partial x_1)^{m_1}\ldots (\partial /\partial x_\ell )^{m_\ell}$.

Let $\epsilon_0>0$ be so small that
$$\{z_1Z_1+\ldots +z_\ell Z_\ell | \, |z_j|<\epsilon_0\,\, \mathrm{for}\,\,
j=1,\ldots ,\ell\}\subset V\, .$$
Then $f_\lambda$ is holomorphic on $\{z=(z_1,\ldots ,z_\ell )\mid |z_j|<\epsilon_0
\,\, \mathrm{for}\,\,
j=1,\ldots ,\ell\}$. By
Cauchy's integral theorem for the derivatives of $f_\lambda$ we get
for each $\epsilon< \epsilon_0$
$$\partial^m f_\lambda (0)=\frac{m!}{(2\pi i)^{\ell}}\oint_{|z_1|=\epsilon} \cdots  \oint_{|z_\ell |=\epsilon}
 \frac{f_\lambda (\zeta )}{\zeta^{m+1}}\, d\zeta_1\ldots d\zeta_\ell\, .$$
Thus (\ref{eq-t:estimate1}) implies, with the same constant
$C$ as in (\ref{eq-t:estimate1}), that
\begin{eqnarray*}
\left|\partial^m f_\lambda (0)\right|
&\le & C m!(2\pi )^{-\ell}\epsilon^{-(|m|+\ell)}e^{|X| |\Re \lambda|}
e^{\epsilon \ell  (|\Re \lambda |+   |\Im \lambda|)} (2\pi \epsilon )^\ell \\
&=& Cm!
e^{\epsilon \ell  (|\Re \lambda |+   |\Im \lambda|)} \epsilon^{-|m|}e^{|X| |\Re \lambda|} \, .
\end{eqnarray*}
Now, take
$$\epsilon =\frac{\epsilon_0}{\ell (1+|\lambda | )}$$
then, with a new constant $C$ depending on $V$, but
independent of $\lambda$ and $X$, we get
\[
\left| \partial^m f_\lambda (0)\right|\le
C(1+|\lambda |)^{|m|}e^{|X| |\Re \lambda|}
\]
as was to be shown.
\end{proof}

\section{Paley-Wiener Theorem for Distributions}\label{s:distributions}
\noindent
In this section we state and prove the Paley-Wiener theorem for
distributions on $M$.
\begin{definition}
\label{d: PW-dist-space}
(Paley-Wiener space for distributions)
For $r>0$ let
$\PW_r^*(\fa)$
denote the space of holomorphic functions $\Phi$ on $\fa_\C^*$
satisfying the following.
\begin{itemize}
\item[a)] There exists a $k\in\Z^+$ and a constant $C_k>0$ such that
$$|\Phi (\gl)|\leq C_k(1+|\gl|)^{k} e^{r|\Re\gl |}$$
for all $\gl\in\fa_\C^*$.

\item[b)] $\Phi (w(\gl+\rho)-\rho)=\Phi(\gl)$ for
all $w\in W$, $\gl\in\fa_\C^*$.
\end{itemize}
\end{definition}

Let $r>0$. A distribution $F$ has
support in $\bar D_r(o)$ if and only
if $F(f)=0$ for all $f\in C^\infty (M)$ with
$\Supp (f)\subset M\setminus \bar D_r(o)$.
Denote by $C_r^{-\infty}(M)$ the space of distributions
that are supported on $\bar D_r(o)$.

\begin{remark}\label{dadokerror}
Recall (see (\ref{eq-dist})) that every distribution $F$
on $M$ satisfies an estimate
$$|F(f)|\leq C \sup_{x\in M, j\leq k} |\Delta_M^j f(x)|\, .$$
If the
support of $F$ is contained in some compact subset $S\subset M$,
it is tempting to replace the supremum over $x\in M$ by the supremum
over $x\in S$, but in general such an estimate is false. The
supremum \textit{has} to be taken over an open neighborhood of $S$
(see \cite{Schwartz}, example p. 95 and discussion p. 98-100).
This causes a minor complication in the proof of Theorem
\ref{t: PW-d}
(this problem appears to be overlooked in \cite{Dadok}).
\end{remark}

We need the following elementary result.

\begin{lemma}\label{le-stonglyHolo} Let
$\Omega\subset\C^\ell$ be open and let $M$
be a differentiable manifold.
Let $f \in C^\infty (\Omega\times M)$,
and assume that $f$ is holomorphic in the first variable.
Then $z \mapsto f(z,\cdot )$
is holomorphic as a map $\Omega\to C^\infty (M)$.
\end{lemma}
\begin{proof}
We first observe that for $a\in\Omega\subset\real^{2\ell}$
and $f \in C^\infty (\Omega\times M)$, we have
$$\frac {f(a+he_j,\cdot)-f(a,\cdot)}h\to
\frac{\partial f}{\partial x_j}(a,\cdot)$$
in $C^\infty(M)$ for $h\to 0$ and $j=1,\dots,2\ell$.
Hence, if $T$ is a continuous linear form on $C^\infty(M)$,
it follows that $a\mapsto T(f(a,\cdot))$
is differentiable on $\Omega$ with
$$\frac\partial{\partial x_j} [T(f(a,\cdot))]=
T( \frac{\partial f}{\partial x_j}(a,\cdot) ).$$

It follows from this observation that $z\mapsto T(f(z,\cdot))$
is continuously differentiable and satisfies the Cauchy-Riemann
equations, for each continuous linear form $T$ on $C^\infty(M)$.
Hence $z \mapsto f(z,\cdot )$ is weakly holomorphic into
$C^\infty(M)$, and, as
this space is Fr\'echet, also strongly holomorphic.
\end{proof}

\begin{theorem}[Local Paley--Wiener theorem for distributions]
\label{t: PW-d}
There exists $R>0$ such that the following holds
for each $0<r<R$.

\begin{itemize}
\item[i)] Let $F\in C^{-\infty}_r(M)^K$.
Then the Fourier transform $\tilde{F} \colon \Lambda^+(M)\to\C$
extends  to a function in $\PW_r^*(\fa)$.

\item[ii)] Let $\Phi \in\PW_r^*(\fa)$. There exists
a unique distribution $F\in C^{-\infty}_r(M)^K$ such that
$\tilde{F} (\mu)=\Phi(\mu)$ for all $\mu\in \Lambda^+(M)$.

\item[iii)] The functions in the Paley-Wiener space $\PW_r^*(\fa )$
are uniquely determined by their values on $\Lambda^+(M)$.
\end{itemize}
Thus, the Fourier transform followed by the extension gives a bijection
$$C_r^{-\infty}(M )^K\simeq \PW_r^*(\fa )\, .$$
\end{theorem}

\begin{remark}
Note, that as in Theorem \ref{t: PW}, $R$ can be different
in each part of the above theorem.
\end{remark}

\begin{proof}
(i) Let $R>0$ be such that ${D_R(o)}\subset
K\exp\Omega\cdot o$, where $\Omega$ is defined in
(\ref{Omega}).
Let $r<R$ and let $\epsilon>0$ be so that $r+\epsilon <R$.
Let  $\varphi\in C^\infty (M)^K$ be a function
which is 1 on a neighborhood of
the closed ball $\bar D_r(o)$, and
supported on
$\bar D_{r+\epsilon}(o)$.
The product $\varphi\psi_\lambda$ is a globally
defined smooth function on $M$, and it belongs to $C^\infty(M)^K$
for all $\gl\in\fa_\C^*$.
Let $F\in C_r^{-\infty }(M)^K$.
We extend the Fourier transform
of $F$ to a function on $\fa_\C^*$ by
\begin{equation}\label{eq-extension}
\widetilde{F}(\lambda ):=F(\varphi \psi_\lambda^\vee)\, .
\end{equation}
The extension is independent of the choice of $\varphi$. Note
also that
$\widetilde{F}(w(\lambda+\rho)-\rho)=\widetilde{F}(\lambda)$
and that
\begin{equation*}
\widetilde{F}(\lambda )=
F(\varphi \varphi_{\lambda+\rho}^\vee )
=F (\varphi \varphi_{-\lambda -\rho})
=F(\varphi \psi_{-\lambda -2\rho})\, .
\end{equation*}

Since the map $(\lambda,x)\mapsto \psi_\lambda(x)$
is smooth on the open subset $\fa_\C^*\times \exp(\Omega)\cdot o$
of $\fa_\C^*\times A\cdot o$, it follows that
$(\lambda,x)\mapsto \varphi(x)\psi_\lambda(x)$
is smooth on $\fa_\C^*\times A\cdot o$.
By Lemma \ref{le-stonglyHolo} it follows
that $\lambda \mapsto \varphi \psi_{\lambda}$ is
holomorphic into $C^\infty(A\cdot o)$,
and as it is also $W$-invariant in the $A$-variable,
it follows from Lemma \ref{le-restriction}
that it is holomorphic into $C^\infty(M)^K$.
Hence $\lambda\mapsto \widetilde{F}(\lambda)$ is
holomorphic on $\fa^*_\C$.

We still need to show that this extension has exponential
growth with exponent $r$.
For that we will choose the function $\varphi$
of (\ref{eq-extension}) in such a way
that we can control the right hand side of equation
(\ref{eq-widetildeF}) below (this is similar to
what is done in the Euclidean case, see for example
\cite{Horm}).

As $F$ is a $K$-invariant distribution, it
follows by Lemma  \ref{le-KinvariantDistribution} that there exists
finitely many $u_1,\dots, u_s\in U(\fa )$ and a constant
$C>0$, such that
\begin{equation}\label{eq-widetildeF}
|\widetilde{F}(\lambda )|\leq C \max_{i=1,\dots,s}
\|L_{u_i} (\varphi|_{A\cdot o} \psi_{\lambda}^\vee |_{A\cdot a} )\|_\infty \, .
\end{equation}

Let $m=\max_{i=1,\dots,s}\deg u_i$.
Let $h:\R\to \R$ be such that $0\le h\le 1$,
$h|_{(-\infty, 1/3]}=1$, and $\Supp (h)\subseteq (-\infty , 2/3]$.
Let $C>0$ be such that $\|h^{(j)}\|_\infty \le C$ for $j=0,\ldots
,m$ (where $m$ is as above).
Finally,
for $\delta>0$ let
$h_\delta (t)=h(t/\delta)$. Then $h_\delta$ has the properties that

\begin{enumerate}
\item $0\le h_\delta \le 1$,
\item $h_\delta (t)=1$ for all $t\le \delta/3$,
\item $h_\delta (t)=0$ if $2\delta/3 \le t$,
\item $|h_\delta^{(j)}(t)|\le C\delta^{-j}$ for all $t\in\R$,
$j=0,1,\ldots ,m$  and $\delta>0$.
\end{enumerate}

Recall that $r+\epsilon<R$ and
let  $\delta\leq \gep$ be arbitrary for the moment.
Then $r+\delta<R$. Let
$$\varphi (x)=h_\delta (d(x,o)-r)$$
for $x\in M$.
Then
$\Supp (\varphi)\subset \bar D_{r+\epsilon}(o)$ and
$\varphi=1$ on a neighborhood of $\bar D_r(o)$.
Let $j\leq m$ and let
$X_1,\ldots ,X_j\in \fa$ with $|X_i|=1$. By applying the
chain and Leibniz rules we obtain
\begin{equation}
|X_1\ldots X_j \varphi (x)|\le C_1\delta^{-j}
\label{e:estimate1}
\end{equation}
for some constant $C_1>0$.
Note that $C_1$ is independent of $\delta$.
In fact it  only depends  on the
constant $C$ above, and the derivatives of $x\mapsto d(x,o)$ on the compact
set  $\{x\in M\mid r\le d(x,o)\le r+\gep\}$.
As $d(\cdot ,o)$ is smooth away from
$o$ it follows that those derivatives are bounded independently of $\delta$.


For the derivatives $X_1\ldots X_j \psi_{\lambda}^\vee $ we note first that
$|\Re (-\lambda -2\rho)|=
|\Re \lambda|$ as $\rho \in i\fa^*$.
By Corollary \ref{l:estimate derivatives} we get for $X\in \fa$, $|X|<R$:
\begin{equation}\label{e:estimate2}
|X_1\ldots X_j\psi_\lambda^\vee (\Exp X)|=D(1+|\lambda |)^je^{(r+\delta )|\Re (\lambda)|}
\end{equation}
for some constant $D$, independent of $\lambda$.

Using (\ref{eq-extension}),
the estimates (\ref{eq-widetildeF}), (\ref{e:estimate1}),
(\ref{e:estimate2}), and the Leibnitz rule,  it follows that there
exists a constant $C>0$ such that
for all $\lambda
\in\fa_\C^*$ and every $\delta\leq\eps$ we have
$$|\widetilde{F}(\lambda )|\le C\delta^{-m}(1+|\lambda |)^{m}
e^{(r+\delta)|\Im \lambda |}\, .$$
We now specialize to $\delta=(1+|\gl|)^{-1}\epsilon$ and conclude that
$\widetilde{F}\in \PW_r^*(\fa)$.

(ii) \def\newPsi{\Phi}
Let $\newPsi \in \PW_r^*(\fa)$. The asserted uniqueness
of $F$ follows from Lemma \ref{l:F(f)}.
Motivated by (\ref{e:F(f)}) in that lemma
we define $F : C^\infty (M)^K\to \C$ by
\begin{equation}\label{eq-defF}
F(f):=\sum_{\mu \in \Lambda^+(M)} d (\mu^* ) \tilde{f}(\mu^* )\newPsi (\mu )\, .
\end{equation}
We need to justify the convergence of the sum.
Let $\omega (\lambda )=\langle \lambda ,\lambda +2\rho\rangle$.
Then
$\Delta_M\psi_\mu =-\omega (\mu )\psi_\mu$.
Let $\Lambda_1=\Lambda^+(M)\setminus \{0\}$, and observe that
$\omega(\mu)>0$ for all $\mu\in\Lambda_1$.
Let $D_1>0$ be such that
$$(\forall \mu \in \Lambda_1)\qquad \omega (\mu )
\ge D_1 (1+| \mu |)\, .$$
By Weyl's dimension formula, there exists a constant $D_2>0$ and
$m\in \Z^+$ such
that
$$d (\mu^*) \le D_2(1+|\mu |)^{m}\, .$$
Let $k\in \Z^+$ be such that
$$|\newPsi (\gl)|\leq C (1+|\gl|)^{k} e^{r|\Re\gl|}\, ,$$
and let $s\in \Z^+$ be  such that
$$s>m+k+\ell (U)\, .$$
For the sum over $\Lambda_1$ we use Lemma \ref{le-su} and
the fact that $\Delta_M\psi_\mu = -\omega (\mu )\psi_\mu$ to get:
\begin{eqnarray}
\label{estimate sum}
\sum_{\mu\in\Lambda_1}d (\mu^*)|\widetilde{f}(\mu^*)||\newPsi(\mu )|
&=&
\sum_{\mu\in\Lambda_1}d (\mu^*)\frac{|\widetilde{(\Delta_{M}^s f)}(\mu^*)|}{|\omega (\mu ) |^{s}}
|\newPsi(\mu )|\\
&\le &C_s\left( \sum_{\mu\in \Lambda_1}(1+|\mu |)^{m+k-s}\right)\|\Delta_M^s(f)\|_\infty\\
&<&\infty \, .
\end{eqnarray}
Here $C_s=CD_1D_2$. It follows that
$$
|F (f)|\le |\newPsi (0)| \|f\|_\infty +C_s\|L_M^sf \|_\infty<\infty\, .$$
Thus $F$ is well defined and continuous. It is linear by definition. It follows
that $F\in C^{-\infty}(M)^K$.
Let $\mu \in \Lambda^+ (M)$. By application of (\ref{eq-defF})
to $f=\psi_\mu^\vee$
it follows, using (\ref{Schur}),
that $\widetilde{F}(\mu )=\newPsi (\mu)$.

To finish the proof of (ii), we need to show  $\Supp (F)\subseteq \Drc$.
For each $\epsilon  >0$ such that $r+\epsilon <R$ let $f_\gep\in C^\infty (M)^K$ be
positive with $\Supp (f_\gep)\subseteq \Bec$ and
$\int_{M} f_\gep (x)\, dx=1$. Note
that $|\widetilde{f_\gep}(\mu )|\le 1$ for all $\mu \in \gL^+ (M)$ and
that $\lim_{\gep \to 0} \widetilde{f_\gep}(\mu)\to 1$ for
each $\mu\in\Lambda^+(M)$.

Denote the holomorphic extension of $\widetilde{f_\gep}$
by the same letter and recall that $\widetilde{f_\gep}\in \PW_{\gep}(\fa )$.
Let $\phi_\gep(\lambda ) :=\newPsi(\lambda )\widetilde{f_\gep } (\lambda )$.
Then $\phi_\gep\in \PW_{r+\gep}(\fa )$. By Theorem \ref{t: PW}, part (ii), there
exists $F_\gep \in C_{r+\epsilon}^\infty ( M)^K$ such that
$\widetilde{F_\gep} = \phi_\gep$. In particular
$$\int f(x) F_\gep (x)\, dx=
\sum_{\mu\in\gL^+ (M)}d (\mu^* )\tilde{f}(\mu^* )\widetilde{f_\gep}(\mu )
\newPsi(\mu )$$
for all $f\in C^\infty (M)^K$.
Hence, using (\ref{estimate sum}) to justify the limit,
$$\lim_{\gep \to 0}\int_{M} f(x)F_\gep (x)\, dx\, =F(f).$$
As the support of $F_\gep$ is contained in
$\bar
{D}_{r+\epsilon}(o)$  it follows
that the support of $F$ is contained in
$$\bigcap_{\gep>0}\bar{D}_{r+\epsilon}(o) = \Drc\, .$$

(iii) follows as in the proof of Theorem \ref{t: PW} in \cite{OS},
given in Section 7 of that paper.\end{proof}

Let $F\in C^{-\infty}(M)^K$. The \textit{singular support of $F$},
is the complement of
the largest open
set on which $F$ is given by a smooth function.

\begin{theorem}[Characterization of singular support]\label{t:singsupp}
Let $R$ be as in Theorem \ref{t: PW} and
$0<s\le r<R$. Let $F\in C^{-\infty}_r(M)^K$. Then the singular
support of $F$ is contained in $\Bsc$ if and only if there exists
$N\in \Z$ such that for each $m\in \Z^+$ the holomorphic
extension of $\tilde{F}$ satisfies
$$|\tilde{F}(\lambda)|\le C_m (1+|\lambda |)^N e^{s|\Re \lambda |}$$
for all $\lambda \in \fa_\C^*$ with $|\Re \lambda |\le m \log
(1+|\lambda |)\, .$
\end{theorem}

\begin{proof}
The proof, in which Theorems \ref{t: PW} and
\ref{t: PW-d} play the crucial roles,
is similar to that of Propositions 1.3 and 1.4
in \cite{Dadok}, which is a reduction to H\"ormander's theorem for
the Euclidean case.
\end{proof}

If $D$ is a differential operator on $M$, then we define the differential
operator $D^*$ by
$$\int_M Df(x)\, g(x) \, dx=\int_M f(x) D^* g(x)\, dx$$
for $f,g\in C^\infty (M)$. Recall also the definition of the
Harish-Chandra isomorphism $\gamma^*: \D (M)\to S(\fa^*)^W$ from
Lemma 5.1 in \cite{OS}.

\begin{theorem}[Solvability for distributions]\label{t:solvability} Let
$0<r<R$, let $F\in C^{-\infty}_r (M)^K$, and let
$D\in\D (M)$. Then there exists a  $T\in C^{-\infty}_r (M)^K$
such that
$DT=F$ if and only if $\lambda\mapsto \widetilde{F}
(\lambda )/\gamma (D^*,-\mu -\rho)$
is entire. Moreover, in that case $T$ is unique.
\end{theorem}
\begin{proof} The proof is the same as that of Theorem 1.8, p. 419, in
\cite{GGA}.
\end{proof}

\section{Appendix}
\noindent
In this appendix we discuss the extension to compact symmetric spaces
of the preceding results. We begin by generalizing the results from
\cite{OS}.

Let $U$ be a connected compact Lie group, not necessarily
semisimple. As before, let $\theta$ be an involution and let
$M=U/K$ where $U^\theta_0\subset K\subset U^\theta$.
Let $Z$ denote the center of $U$. We assume that
$Z\cap K=\{e\}$, since otherwise we can replace $U$ by $U/Z\cap K$,
noticing that $Z\cap K$ acts trivially on $U/K$.
Let $\fu$ denote the Lie algebra of $U$, then
$\fu=\fz\oplus \fu^\prime$ where $\fz$ is the center
of $\fu$ and $\fu^\prime=[\fu,\fu]$ is semisimple.
As before we denote by $\theta$ also the induced involution of
$\fu$, and by $\fu=\fk\oplus\fq$ the corresponding Cartan
decomposition. Then $\fz$ and $\fu'$ are $\theta$-invariant,
and it follows from our assumption that
$\fz\subset\fq$ and $\fu'\supset\fk$.
Denote by $Z_0$ and $U'$ the analytic subgroups
of $U$ corresponding to $\fz$ and $\fu'$. Then
$U=Z_0U'$ and $Z_0\cap U'$ is finite.
It follows that $Z_0\times U'$ is a covering of $U$ by the homomorphism
$(z,u)\mapsto zu$. The
kernel is $D=\{(z,z^{-1})\mid z\in Z_0\cap U'\}$.
Thus the covering is
$$Z_0\times U' \to U\simeq (Z_0\times U')/D\, .$$
The identity component $K_0$ of $K$ is contained in $U'$,
hence the subgroup $K'=U'\cap K$ is a symmetric subgroup of $U$.
In general $K'$ can be a proper subgroup of $K$, in spite of
the assumption that $Z\cap K=\{e\}$. Let
$K^\times\subset Z_0\times U'$ denote the preimage of $K$,
then it follows that the covering above induces a covering map
$$Z_0\times U'/K' \to U/K\simeq (Z_0\times U')/K^\times \, . $$

We fix a maximal abelian subspace $\fa$ of $\fq$,
then $\fz\subset\fa$. As before, $\Sigma\subset i\fa^*$
denotes the sets of restricted roots, and $\Sigma^+$
denotes a subset of positive roots. We donote by
$\Lambda^+(U/K)\subset i\fa^*$ the set of highest weights
of irreducible $K$-spherical representations.
It is now seen that Lemma \ref{semilattice}
is valid in the generalized situation too: For $Z_0\times U'/K'$
this is a straightforward extension, otherwise we
repeat the proof of the lemma, replacing $K^*$ with $K^\times$.

We define the spherical Fourier transform of a $K$-invariant
function on $U/K$ as before, and
for each $r>0$ we define the Paley-Wiener space $\PW_r(\fa)$
exactly as in Definition \ref{d: PW space}.
In particular,
$W$ is the Weyl group associated with the root system $\Sigma$.
We can then state:

\begin{theorem}
\label{t: PW extended}
Let $U/K$ be a compact Riemannian symmetric space with
assumptions as described above. Then Theorem \ref{t: PW} is
valid exactly as stated in Section \ref{s: main thm}.
\end{theorem}

\begin{proof} For $Z_0\times U'/K'$ this is a straightforward
extension of the proof given in \cite{OS}. For the general case
we apply the covering map above. Here it is used that every smooth
function $f$ on $U/K$ supported on a sufficiently small
$K$-invariant neighborhood of $eK$ lifts to a smooth function
$F$ on the cover $Z_0\times U'/K'$, supported in a
$K'$-invariant neighborhood of $eK'$ of the same size.
The lifted function $F$ is $K'$-invariant if and only if $f$
is $K$-invariant, and the
Fourier transform $\tilde F$ of the lifted function restricts to the
Fourier transform $\tilde f$ of the original function on
$\Lambda^+(U/K)\subset \Lambda^+(Z_0\times U'/K')$.
Noticing that by definition $\PW_r(\fa)$ is the same space
in the two cases $U/K$ and $Z_0\times U'/K'$, we thus have
a commutative diagram of bijective maps
$$
\begin{array}{ccc}
C^\infty_r(Z_0\times U'/K')^{K'}&\to&    \PW_r(\fa)\\
\uparrow&&\parallel\\
C^\infty_r(U/K)^K&\to&                   \PW_r(\fa)
\end{array}
$$
for $r$ sufficiently small.
The horisontal arrows represent Fourier transform followed by
holomorphic extension, and it follows from Lemma \ref{semilattice}
by the argument in \cite{OS} Section 7, that functions in
$\PW_r(\fa)$ are uniquely determined by their restriction to
$\Lambda^+(U/K)$. The theorem is now easily proved.
\end{proof}

The main results of the present paper, Theorems
\ref{t: PW-d} and
\ref{t:singsupp}, can now be generalized to the present setting by
a straightforward extension of the previous proof. We omit the
details.

\def\adritem#1{\hbox{\small #1}}
\def\distance{\hbox{\hspace{3.5cm}}}
\def\addgestur{\vbox{
\adritem{G.~\'Olafsson}
\adritem{Department of Mathematics}
\adritem{Louisiana State University}
\adritem{Baton Rouge}
\adritem{LA 70803}
\adritem{USA}
\adritem{E-mail: olafsson@math.lsu.edu}
}
}
\def\addhenrik{\vbox{
\adritem{H.~Schlichtkrull}
\adritem{Department of Mathematical Sciences}
\adritem{University of Copenhagen}
\adritem{Universitetsparken 5}
\adritem{2100 K\o benhavn \O}
\adritem{Denmark}
\adritem{E-mail: schlicht@math.ku.dk}
}
}
\mbox{\ }
\vfill
\hbox{\vbox{\addgestur}\vbox{\distance}\vbox{\addhenrik}}

\end{document}